%% file: MCP-Htype-v4.tex
\documentclass[11pt,a4paper]{amsart}

\usepackage[bookmarksopen]{hyperref}

\usepackage{xcolor}
\hypersetup{
    colorlinks,
    linkcolor={red!50!black},
    citecolor={blue!50!black},
    urlcolor={blue!80!black}
}

\usepackage[utf8]{inputenc}
\usepackage[english]{babel}
\usepackage{amsmath,amsthm,amsfonts,amssymb}
\usepackage{lmodern}

\usepackage{mathtools}

\theoremstyle{plain}
\newtheorem{theorem}{Theorem}
\newtheorem*{theorem*}{Theorem}
\newtheorem{prop}[theorem]{Proposition}
\newtheorem{cor}[theorem]{Corollary}
\newtheorem{lemma}[theorem]{Lemma}
\newtheorem*{lemma*}{Lemma}
\newtheorem*{fact*}{Fact}
\theoremstyle{definition}
\newtheorem{definition}[theorem]{Definition}
\theoremstyle{remark}
\newtheorem{example}[theorem]{Example}
\newtheorem{rmk}[theorem]{Remark}

\DeclareRobustCommand{\gobblefive}[5]{}							
\newcommand*{\SkipTocEntry}{\addtocontents{toc}{\gobblefive}}

\usepackage{bm}	% per i caratteri calligrafici in bold
\DeclareMathAlphabet{\mathbbold}{U}{bbold}{m}{n}	% nuovo alfabeto per i numeri in bold.

\usepackage[left=3cm,right=3cm,top=3cm,bottom=3cm]{geometry}

\usepackage[textwidth=2.2cm]{todonotes}

\DeclareMathOperator{\Tr}{\mathrm{Tr}}

\newcommand{\ball}{B}
\newcommand{\R}{\mathbb{R}}

\newcommand{\distr}{\mathcal{D}}

\usepackage[foot]{amsaddr}

\title[Sharp measure contraction property for generalized H-type groups]{Sharp measure contraction property for generalized H-type Carnot groups}

\author[Davide Barilari]{Davide Barilari$^\flat$}
\address{$^\flat$Institut de Math\'ematiques de Jussieu-Paris Rive Gauche UMR CNRS 7586, Universit\'e Paris-Diderot,
Batiment Sophie Germain, Case 7012, 75205 Paris Cedex 13, France}
\email{\href{mailto:davide.barilari@imj-prg.fr}{davide.barilari@imj-prg.fr}}

\author[Luca Rizzi]{Luca Rizzi$^\sharp$}
\address{$^\sharp$Univ. Grenoble Alpes, CNRS, Institut Fourier, F-38000 Grenoble, France}
\email{\href{mailto:luca.rizzi@univ-grenoble-alpes.fr}{luca.rizzi@univ-grenoble-alpes.fr}}

\subjclass[2010]{53C17, 53C22, 35R03, 54E35, 53C21}
\keywords{sub-Riemannian geometry, Carnot groups, measure contraction property, geodesic dimension, optimal transport}

\date{\today}
\begin{document}

\begin{abstract}
We prove that H-type Carnot groups of rank $k$ and dimension $n$ satisfy the $\mathrm{MCP}(K,N)$ if and only if $K\leq 0$ and $N \geq k+3(n-k)$. The latter integer coincides with the geodesic dimension of the Carnot group. The same result holds true for the larger class of \emph{generalized H-type} Carnot groups introduced in this paper, and for which we compute explicitly the optimal synthesis. This extends the results of \cite{MCPc1} and constitutes the largest class of Carnot groups for which the curvature exponent coincides with the geodesic dimension. We stress that generalized H-type Carnot groups have step 2, include all corank 1 groups and, in general, admit abnormal minimizing curves.

As a corollary, we prove the absolute continuity of the Wasserstein geodesics for the quadratic cost on all generalized H-type Carnot groups.
\end{abstract}

\maketitle
\setcounter{tocdepth}{1} %guarda se ti piace cosi'
\tableofcontents

\input{introduction.tex}

\input{preliminaries.tex}
\input{proofs.tex}

\SkipTocEntry\subsection*{Acknowledgements}
\thanks{This research has been supported by the Grant ANR-15-CE40-0018 of the ANR. This research benefited from the support of the ``FMJH Program Gaspard Monge in optimization and operation research'' and from the support to this program from EDF. This work has been partially supported by the ANR project ANR-15-IDEX-02.}

\bibliographystyle{abbrv}
\bibliography{MCP-Htype}

\end{document}

%% file: introduction.tex
\section{Introduction}

The measure contraction property (MCP) is one of the generalizations of Ricci curvature bounds to non-smooth metric measure spaces $(X,d,\mu)$.
%Among the different generalizations of Ricci curvature bounds that have been introduced in the non-smooth setting of metric measure spaces $(X,d,\mu)$, the measure contraction property (MCP) is one of the weakest.
The condition $\mathrm{MCP}(K,N)$ we discuss here was introduced by Ohta \cite{Ohta-MCP} and controls how the measure $\mu$ of a set in $X$ is distorted along geodesics defined by the distance $d$, where $(X,d)$ is assumed to be a length space (see also Sturm \cite{S-II}, for a slightly stronger version of  the same condition).
 
When $X$ is a complete Riemannian manifold of dimension $N\geq 2$, $d$ is the Riemannian distance, and $\mu$ is the Riemannian measure, the condition $\mathrm{MCP}(K,N)$ is equivalent to require that $X$ has Ricci curvature bounded from below by $K$ (see \cite[Thm. 3.2]{Ohta-MCP}). The measure contraction property is thus a synthetic replacement for Ricci curvature bounds on metric spaces. We refer to Section \ref{s:mms} for a precise definition. We recall that the $\mathrm{MCP}$ is the weakest among the synthetic curvature conditions introduced in \cite{S-I,S-II,LV-Ricci} and developed in the subsequent literature. We stress that all these conditions are stable under pointed Gromov-Hausdorff limits.
 
\medskip
In this paper, we continue the investigation of measure contraction properties in sub-Riemannian geometry, started in \cite{Juillet,Rifford,MCPc1}.  More precisely we focus our attention on a class of Carnot groups that we call \emph{generalized H-type groups}, which includes the classical H-type Carnot groups, in the sense of Kaplan~\cite{Kaplan,BLU-Stratified}. Here a Carnot group $G$ is considered as a metric measure space $(G,d,\mu)$ equipped with the Carnot-Carath\'eodory distance $d$ and a left-invariant measure $\mu$ (thus a multiple of Haar and Popp measures, see \cite{montgomerybook,BR-popp}).

Even though Carnot Groups (and more in general sub-Riemannian manifolds) can be seen as Gromov-Hausdorff limits of sequences of Riemannian ones with the same dimension, these sequences have Ricci curvature unbounded from below (for an explicit example, see \cite{Rifford}). Thus, it is not clear whether these structures satisfy some $\mathrm{MCP}$. 
   
The simplest Carnot groups, that is the $2d+1$ dimensional Heisenberg groups $\mathbb{H}_{2d+1}$, do not satisfy the $\mathrm{MCP}(K,2d+1)$ for any $K$. Indeed, in \cite{Juillet}, Juillet proved that $\mathbb{H}_{2d+1}$ satisfies 
$\mathrm{MCP}(K,N)$ if and only if $K \leq 0$ and $N \geq 2d+3$. The number $2d+3$ is then the optimal dimension for the synthetic condition $\mathrm{MCP}(0,N)$, and is larger than both the topological ($2d+1$) and the Hausdorff ($2d+2$) dimensions of $\mathbb{H}_{2d+1}$.

Juillet's results have been extended in \cite{MCPc1}, where it is proved that any corank $1$ Carnot group $(G,d,\mu)$ of rank $k$ satisfies the $\mathrm{MCP}(K,N)$ if and only if $K \leq 0$ and $N \geq k+3$. In particular, they satisfy the $\mathrm{MCP}(0,k+3)$.

For what concerns more general structures, in \cite{Rifford}, Rifford proved that if a Carnot group of rank $k$ and dimension $n$ is ideal (i.e.\ does not admit abnormal minimizing curves), then it satisfies $\mathrm{MCP}(0,N)$ for some finite $N$ satisfying the inequality $N\geq Q+n-k$, where $Q$ is the Hausdorff dimension of the Carnot group. For a non-ideal Carnot group, if one \emph{assumes} that $\mathrm{MCP}(0,N)$ is satisfied for some $N$, the inequality $N\geq  Q+n-k$ still holds.

\subsection{Geodesic dimension}
Since $\mathrm{MCP}(K,N)$ implies $\mathrm{MCP}(K,N')$ for all $N'\geq N$, the above results raised the question of which is the optimal $N$ such that the $\mathrm{MCP}(K,N)$ is satisfied on a general metric measure space. A bound is given by the concept of geodesic dimension, introduced in \cite{ABR-Curvature} for sub-Riemannian structures and in \cite{MCPc1} for metric measure spaces, and which we briefly overview here (see Section \ref{s:mms} for details). Let $(X,d,\mu)$ be a metric measure space, which for simplicity we assume having negligible cut loci. Let $\Omega \subset X$ be a bounded, measurable set, with $0<\mu(\Omega)<+\infty$, and consider its homothety $\Omega_t$ with center $x$ and ratio $t \in [0,1]$, such that $\Omega_{1}=\Omega$ and $\Omega_{0}=\{x\}$. The \emph{geodesic dimension} $\mathcal{N}(x)$ at the point $x \in X$ is, roughly speaking, the largest integer such that $\mu(\Omega_t) = O(t^{\mathcal{N}(x)})$ when $t \to 0$, for all such $\Omega$.

\begin{prop}[\cite{MCPc1}]\label{t:bound}
Let $(X,d,\mu)$ be a metric measure space with negligible cut loci that satisfies the $\mathrm{MCP}(K,N)$, for some $K \in \R$ and $N >1$ or $K\leq 0$ and $N=1$. Then
\begin{equation}
N \geq \sup \{\mathcal{N}(x) \mid x \in X\}.
\end{equation}
where $\mathcal{N}(x)$ is the geodesic dimension of $X$ at $x$.
\end{prop}

\subsection{Carnot groups}
In the case of left-invariant structures on Carnot (or, more generally, Lie) groups, the number $\mathcal{N}(x)$ is independent on $x\in X$, and we denote it by $\mathcal{N}$. 
Furthermore, notice that spaces verifying $\mathrm{MCP}(K,N)$ with $K>0$ are compact, so that Carnot groups can satisfy the $\mathrm{MCP}(K,N)$ only if $K \leq 0$. On the other hand, thanks to the homogeneity properties of Carnot groups, if $\mathrm{MCP}(K,N)$ is satisfied for some $K\leq 0$, then $\mathrm{MCP}(\varepsilon K,N)$ is also satisfied, for all $\varepsilon \geq 0$. For this reason the $\mathrm{MCP}(0,N)$ plays a special role for Carnot groups, which motivated the definition of \emph{curvature exponent}, as
\[
N_0:= \inf\{N>1 \mid \mathrm{MCP}(0,N) \text{ is satisfied}\}.
\]
Indeed, by Proposition~\ref{t:bound}, $N_0 \geq \sup\{ \mathcal{N}(x) \mid x \in X\}$. The following example shows that, in general, the curvature exponent $N_{0}$ can be strictly bigger than the geodesic dimension.
\begin{example}[Riemannian Heisenberg]\label{ex:ex1}
Consider the left-invariant Riemannian structure on $\R^{3}$ generated by the following global orthonormal vector fields:
\begin{equation}
X = \partial_x -\frac{y}{2}\partial_z, \qquad Y = \partial_y + \frac{x}{2}\partial_z, \qquad Z = \partial_z.
\end{equation}
As for any Riemannian structure, $k=n=\mathcal{N} = 3$. In \cite{Rifford} it is proved that, when equipped with the Riemannian volume, this structure satisfies the $\mathrm{MCP}(0,5)$. The same techniques proves that the $\mathrm{MCP}(0,5-\varepsilon)$ is violated for any $\varepsilon>0$, so its curvature exponent is $N_0 = 5$. We stress that this is a left-invariant structure, but is not a Carnot group.
\end{example}

We expect, however, that for Carnot groups the geodesic dimension coincides with the curvature exponent (cf.\ the conjecture formulated in \cite{MCPc1}). Our main result is a proof of this fact for generalized H-type groups (see Section~\ref{s:prel}), which generalizes the results of  \cite{Juillet,MCPc1}, and constitutes then the largest class of structures for which the aforementioned conjecture is verified.

\begin{theorem}\label{t:htype}
Let $(G,d,\mu)$ be a generalized H-type Carnot group of rank $k$ and dimension $n$. Then it satisfies the $\mathrm{MCP}(K,N)$ if and only if $K \leq 0$ and $N \geq k+3(n-k)$, the latter integer being the geodesic dimension of the Carnot group.
\end{theorem}

We observe that generalized H-type groups may have non-trivial abnormal minimizing curves (albeit not strictly abnormal ones). Equivalently, not all generalized H-type groups are ideal, or fat (see Definitions \ref{d:ideal} and \ref{d:fat}).

As a byproduct of the proof, we describe explicitly normal and abnormal geodesics of generalized H-type groups (Lemmas~\ref{l:geods}, \ref{l:abnormals}), the injectivity domain of the exponential map (Lemma \ref{l:ingjd}), and an explicit formula for its Jacobian determinant (Lemma~\ref{l:jacobian}). In particular, this yields the optimal synthesis of generalized H-type groups.

We mention that in \cite{BR-grande1} the Jacobian estimates of this paper have been used to prove a sharp geodesic Brunn-Minkowski inequality for generalized H-type groups, which as a particular case recovers the measure contraction properties of Theorem~\ref{t:htype}.

\subsection{Consequences for optimal transport}

 Theorem~\ref{t:htype} has consequences on the optimal transportation problem on generalized H-type groups, which we briefly outline here. 

The well posedness of the Monge problem with cost given by the squared sub-Rieman\-nian distance for classical H-type groups was proved in \cite{Rigot2}, generalizing the previous results for the Heisenberg group obtained in \cite{Rigot1}. The same result for general sub-Riemannian structures whose squared distance is Lipschitz in charts on $X\times X$ was obtained in \cite{AL-transport}. The Lipschitzness assumption on the diagonal was finally removed in \cite{FR}. We refer to \cite{rifford2014sub} for background on the transportation problem in sub-Riemannian geometry.

For all step $2$ sub-Riemannian structures on a smooth manifold $X$ (such as generalized H-type Carnot groups), $d^2$ is locally Lipschitz on $X\times X$. Hence, according e.g.\ to \cite[Theorem 3.2]{FR}, for any pair of compactly supported probability measures $\mu_0,\mu_1 \in \mathcal{P}_{\text{comp}}(X)$, with $\mu_0$ absolutely continuous with respect to a reference measure $\mu$, there exists a unique optimal transport map for the Monge problem with quadratic cost. 

In this setting, an interesting problem is to determine whether the Wasserstein geodesic $\mu_t \in \mathcal{P}(X)$, interpolating between $\mu_0$ and $\mu_1$, is absolutely continuous w.r.t.\ $\mu$. An affirmative answer for the $2d+1$ dimensional Heisenberg group was the main result of \cite{FJ}. More general structures have been discussed in \cite[Theorem 3.5]{FR}. There it is proved that, if the sub-Riemannian squared distance is also locally semiconcave outside of the diagonal, then $\mu_t \ll \mu$ for all $t \in [0,1)$. Unfortunately, in presence of non-trivial abnormal minimizers, $d^2$ is not necessarily locally semiconcave outside of the diagonal. Absolute continuity of $\mu_t$ must then be proved directly (see, e.g.\ \cite{BKS17} for the case of corank 1 Carnot groups).

As a consequence of Theorem~\ref{t:htype}, and the fact that step $2$ Carnot groups do not admit branching minimizing geodesics, we can apply \cite[Theorem 1.1]{CM}\footnote{See also the appendix of \cite{CM}, where it is discussed how, under the non-branching assumption, Ohta's $\mathrm{MCP}$, which we use here, is equivalent to Sturm's one, which is the one primarily used in \cite{CM}.}. In particular, we extend the absolute continuity result of \cite{FJ} to a setting with non-trivial abnormal minimizers and higher corank. We refer to \cite{CM} for definitions.
\begin{cor}[\cite{CM}]
Let $(X,d,\mu)$ be a generalized H-type Carnot group. Let $\mu_0,\mu_1 \in \mathcal{P}_2(X)$ with $\mu_0 \ll \mu$. Then there exists a unique $\nu \in \mathrm{OptGeo}(\mu_0,\mu_1)$; it is induced by a map $G: X \to \mathrm{Geo}(X)$, and the unique Wasserstein geodesic  between $\mu_0$ and $\mu_1$ satisfies $(e_t)_\sharp \nu =\rho_t \mu \ll \mu$ for all $t \in [0,1)$. Finally, if $\mu_0,\mu_1$ have bounded support and $\rho_0$ is $\mu$-essentially bounded, then
\begin{equation}
\|\rho_t\|_{L^\infty(X,\mu)} \leq \frac{1}{(1-t)^{k+3(n-k)}} \|\rho_0\|_{L^{\infty}(X,\mu)}, \qquad \forall t \in [0,1).
\end{equation}
\end{cor}

%% file: preliminaries.tex
\section{Preliminaries}\label{s:prel}
We recall some basic facts in sub-Riemannian geometry. For details, see \cite{nostrolibro,rifford2014sub,montgomerybook}.

\subsection{Basic definitions} 
A sub-Rieman\-nian manifold is a triple $(M,\distr,g)$, where $M$ is a smooth, connected manifold with $\dim M=n \geq 3$, $\distr$ is a vector distribution of constant rank $k \leq n$ and $g$ is a smooth metric on $\distr$. We  assume that the distribution is bracket-generating. A \emph{horizontal curve} $\gamma : [0,1] \to M$ is a Lipschitz continuous path such that $\dot\gamma(t) \in \distr_{\gamma(t)}$ for almost every $t$, and we define its \emph{length} as follows
\begin{equation}
\ell(\gamma) = \int_0^1 \sqrt{g(\dot\gamma(t),\dot\gamma(t))}dt.
\end{equation}
The \emph{sub-Rieman\-nian (or Carnot-Carath\'eodory) distance} is defined by:
\begin{equation}\label{eq:infimo}
d(x,y) = \inf\{\ell(\gamma)\mid \gamma(0) = x,\, \gamma(1) = y,\, \gamma \text{ horizontal} \}.
\end{equation}
By Chow-Rashevskii theorem, the bracket-generating condition implies that $d: M \times M \to \R$ is finite and continuous. If the metric space $(M,d)$ is complete, then for any $x,y \in M$ the infimum in \eqref{eq:infimo} is attained. In place of the length $\ell$, we consider the \emph{energy functional} 
\begin{equation}
J(\gamma) = \frac{1}{2}\int_0^1 g(\dot\gamma(t),\dot\gamma(t)) dt.
\end{equation}
On the space of horizontal curves defined on a fixed interval and with fixed endpoints, the minimizers of $J$ coincide with the minimizers of $\ell$ parametrized with constant speed. Since $\ell$ is invariant by reparametrization (and we can always reparametrize horizontal curves in such a way that they have constant speed), we do not loose generality in defining \emph{geodesics} as horizontal curves that locally minimize the energy between their endpoints.

The \emph{Hamiltonian} of the sub-Riemannian structure $H : T^*M \to \R$ is defined by
\begin{equation}
H(\lambda) = \frac{1}{2}\sum_{i=1}^k \langle \lambda,X_i\rangle^{2}, \qquad \lambda \in T^*M,
\end{equation}
where $X_1,\dots,X_k$ is any local orthonormal frame for $\distr$. Here $\langle \lambda,\cdot\rangle$ denotes the dual action of covectors on vectors. The Hamiltonian vector field $\vec{H}$ is the unique vector field such that $\sigma(\cdot,\vec{H}) = dH$, where $\sigma$ is the canonical symplectic form of the cotangent bundle $\pi:T^*M \to M$. In particular, the \emph{Hamilton equations} are
\begin{equation}\label{eq:Ham}
\dot\lambda(t) = \vec{H}(\lambda(t)), \qquad \lambda(t) \in T^*M.
\end{equation}
If $(M,d)$ is complete, any solution of \eqref{eq:Ham} can be extended to a smooth curve for all times.

\subsection{End-point map and Lagrange multipliers}

Let $\gamma_u :[0,1] \to M$ be an horizontal curve joining $x$ and $y$. Let $X_1,\dots,X_k$ be a smooth orthonormal frame of horizontal vectors fields, defined in a neighborhood of $\gamma_u$, and let $u \in L^\infty([0,1],\R^k)$ be a \emph{control} such that
\begin{equation}
\dot\gamma_u(t) =  \sum_{i=1}^k u_i(t) X_i(\gamma_u(t)), \qquad \text{a.e. } t \in [0,1].
\end{equation}
Let $\mathcal{U} \subset L^\infty([0,1],\R^k)$ be the neighborhood of $u$ such that, for $v \in \mathcal{U}$, the solution of
\begin{equation}
\dot\gamma_v(t) = \sum_{i=1}^k v_i(t) X_i(\gamma_v(t)), \qquad \gamma_v(0) = x,
\end{equation}
is well defined for a.e. $t \in [0,1]$. We define the \emph{end-point map} with base point $x$ as the map $E_{x}: \mathcal{U} \to M$ that sends $v$ to $\gamma_v(1)$. The end-point map is smooth on $\mathcal{U}$.

We can consider $J : \mathcal{U} \to \R$ as a smooth functional on $\mathcal{U}$ (identifying $\mathcal{U}$ with a neighborhood of $\gamma_u$ in the space of horizontal curves starting from $x$). A minimizing geodesic $\gamma_u$ is a solution of the constrained minimum problem
\begin{equation}
\min\{J(v) \mid  v \in \mathcal{U},\, E_x(v) = y\}.
\end{equation}
By the Lagrange multipliers rule, there exists a non-trivial pair $(\lambda_1,\nu)$, such that
\begin{equation}\label{eq:multipliers}
\lambda_1 \circ D_u E_x  = \nu D_u J, \qquad \lambda_1 \in T_y^*M, \qquad\nu \in \{0,1\},
\end{equation}
where $\circ$ denotes the composition of linear maps and $D$ the (Fr\'echet) differential. If $\gamma_u : [0,1] \to  M$ with control $u \in \mathcal{U}$ is an horizontal curve (not necessarily minimizing), we say that a non-zero pair $(\lambda_1,\nu) \in T_y^*M \times \{0,1\}$ is a \emph{Lagrange multiplier} for $\gamma_u$ if \eqref{eq:multipliers} is satisfied. The multiplier $(\lambda_1,\nu)$ and the associated curve $\gamma_u$ are called \emph{normal} if $\nu = 1$ and \emph{abnormal} if $\nu = 0$. Observe that Lagrange multipliers are not unique, and a horizontal curve may be both normal \emph{and} abnormal. Observe also that $\gamma_u$ is an abnormal curve if and only if $u$ is a critical point for $E_x$. In this case, $\gamma_u$ is also called a \emph{singular curve}. The following characterization is a consequence of the Pontryagin Maximum Principle.
\begin{theorem} \label{t:utile}
Let $\gamma_u :[0,1] \to M$ be an horizontal curve joining $x$ and $y$. A non-zero pair $(\lambda_1,\nu) \in T_y^*M \times \{0,1\}$ is a Lagrange multiplier for $\gamma_u$ if and only if there exists a Lipschitz curve $\lambda(t) \in T_{\gamma_u(t)}^*M$ with $\lambda(1) = \lambda_1,$ such that
\begin{itemize}
\item if $\nu = 1$ then $\dot{\lambda}(t) = \vec{H}(\lambda(t))$, i.e.\ it is a solution of Hamilton equations,
\item if $\nu =0 $ then $\sigma(\dot\lambda(t), T_{\lambda(t)} \distr^\perp) = 0$,
\end{itemize}
where $\distr^\perp \subset T^*M$ is the sub-bundle of covectors that annihilate the distribution.
\end{theorem}
In the first (resp.\ second) case, $\lambda(t)$ is called a \emph{normal} (resp.\ \emph{abnormal}) \emph{extremal}. Normal extremals are   integral curves $\lambda(t)$ of $\vec{H}$. As such, they are smooth, and characterized by their \emph{initial covector} $\lambda = \lambda(0)$. A geodesic is normal (resp.\ abnormal) if admits a normal (resp.\ abnormal) extremal. On the other hand, it is well known that the projection $\gamma_\lambda(t) = \pi(\lambda(t))$ of a normal extremal is locally minimizing, hence it is a normal geodesic. The \emph{exponential map} at $x \in M$ is the map
\begin{equation}
\exp_x : T_x^*M \to M,
\end{equation}
which assigns to $\lambda \in T_x^*M$ the final point $\pi(\lambda(1))$ of the corresponding normal geodesic. The curve $\gamma_\lambda(t):=\exp_x(t \lambda)$, for $t \in [0,1]$, is the normal geodesic corresponding to $\lambda$, which has constant speed $\|\dot\gamma_\lambda(t)\| = \sqrt{2H(\lambda)}$ and length $\ell(\gamma|_{[t_1,t_2]}) = \sqrt{2H(\lambda)}(t_2-t_1)$.

\begin{definition}\label{d:ideal}
A sub-Riemannian structure $(M,\distr,g)$ is \emph{ideal} if the metric space $(M,d)$ is complete and there exists no non-trivial abnormal minimizers.
\end{definition}
\begin{definition}\label{d:fat}
A sub-Riemannian structure $(M,\distr,g)$ is \emph{fat} if for all $x \in M$ and $X \in \distr$, $X(x) \neq0$, then $\distr_x + [X,\distr]_x=T_xM$. 
\end{definition}
The definition of ideal structures appears in \cite{Rifford,rifford2014sub}, in the equivalent language of singular curves. We stress that fat sub-Riemannian structures admit no non-trivial abnormal curves (see \cite[Section 5.6]{montgomerybook}). In particular, complete fat structures are ideal.

\section{Generalized H-type Carnot groups}
A \emph{Carnot group $(G,\star)$ of step $s$} is a connected, simply connected Lie group of dimension $n$, such that its Lie algebra $\mathfrak{g} = T_e G$ is stratified and nilpotent of step $s$, that is
\begin{equation}
\mathfrak{g} = \mathfrak{g}_1 \oplus \dots \oplus \mathfrak{g}_s,
\end{equation}
with
\begin{equation}
[\mathfrak{g}_1,\mathfrak{g}_j] = \mathfrak{g}_{1+j},\quad \forall 1\leq j\leq s, \quad \mathfrak{g}_s \neq \{0\}, \quad \mathfrak{g}_{s+1} =\{0\}.
\end{equation}

Let $\distr$ be the left-invariant distribution generated by $\mathfrak{g}_1$, with a left-invariant scalar product $g$.  This defines a sub-Riemannian structure $(G,\distr,g)$ on the Carnot group. For $x \in G$, we denote with $\tau_x(y) := x \star y$ the left translation. The map $\tau_x : G \to G$ is a smooth isometry.  Any Carnot group, equipped with the Carnot-Carath\'eodory distance $d$ and the Lebesgue measure $\mu$ of $G=\R^n$ is a complete metric measure space $(X,d,\mu)$. Moreover, a Carnot group is fat if and only if it is ideal \cite[Section 8.2]{MCPc1}.

\subsection{Generalized H-type groups}

Let $(G,\star)$ be a step $2$ Carnot group, with Lie algebra $\mathfrak{g}$ of rank $k$ and dimension $n$. In particular $\dim\mathfrak{g}_1=k$, $\dim\mathfrak{g}_2=n-k$, and
\begin{equation}
[\mathfrak{g}_1,\mathfrak{g}_1] = \mathfrak{g}_2, \qquad [\mathfrak{g}_i,\mathfrak{g}_2]=0, \qquad i= 1,2.
\end{equation}
Fix a scalar product $g$ on $\mathfrak{g}$ such that $\mathfrak{g}_1 \perp \mathfrak{g}_2$. The restriction $g|_{\mathfrak{g}_1}$ induces a left-invariant sub-Riemannian structure $(\distr,g)$ on $G$, such that $\distr(p) = \mathfrak{g}_1(p)$ for all $p \in G$. For any $V \in  \mathfrak{g}_2$, the skew-symmetric operator $J_V: \mathfrak{g}_1 \to \mathfrak{g}_1$ is defined by
\begin{equation}
g(X,J_V Y) = g(V,[X,Y]), \qquad \forall X,Y \in \mathfrak{g}_1.
\end{equation}
With an abuse of language we still call $(G,\star)$, equipped with a scalar product $g$ as above and the induced left-invariant sub-Riemannian structure, a step $2$ Carnot group.
\begin{definition} \label{d:genht}
A step $2$ Carnot group is of \emph{generalized H-type} if there exists a symmetric, non-zero and non-negative operator $S : \mathfrak{g}_1 \to \mathfrak{g}_1$ such that
\begin{equation}\label{eq:13}
J_V J_W + J_W J_V = -2 g(V,W) S^2, \qquad \forall V,W \in \mathfrak{g}_2.
\end{equation}
\end{definition}
\begin{rmk}
By polarization, it is easy to show that a step $2$ Carnot group is of generalized H-type if and only if there exists a symmetric, non-negative and non-zero operator $S:\mathfrak{g}_1 \to \mathfrak{g}_1$  such that $J_V^2 = - \|V\|^2 S^2$ for all $V \in \mathfrak{g}_2$.
\end{rmk}
\begin{rmk}
The case $S = \mathrm{Id}_{\mathfrak{g}_1}$ corresponds to classical H-type groups, introduced by Kaplan in \cite{Kaplan}. We refer to \cite[Chapter 18]{BLU-Stratified} for a modern introduction.
\end{rmk}
\begin{rmk}
It is easy to check that corank $1$ Carnot groups, that is step $2$ Carnot groups with $n=k+1$, are always of generalized H-type. Furthermore, if the operator $S$ is non-degenerate (and thus $k=2d$ is even and $S>0$), we are in the case of contact Carnot groups. The eigenvalues of $S$ are also called \emph{frequencies} of the Carnot group (see \cite{ABB-Hausdorff,MCPc1}).
\end{rmk}

\subsection{On existence}
As discussed in \cite[Chapter 18]{BLU-Stratified} and proved in \cite[Corollary 1]{Kaplan}, as a consequence of the Clifford algebra structure enforced by \eqref{eq:13}, classical H-type Carnot groups of rank $k$ and dimension $n$ exist if and only if $n-k \leq \rho(k)$, where $\rho$ is the Hurwitz-Radon function\footnote{$\rho(N)-1$ gives the maximal number of independent vector fields on $\mathbb{S}^{N-1}$.}. In particular, H-type groups of odd rank do not exist (indeed, when $S$ is the identity, $J_V$ must be at the same time skew-symmetric and orthogonal).

The class of generalized H-type groups is larger, and now discuss their existence. Let $0<\lambda_1<\dots<\lambda_\ell$ be the non-zero eigenvalues of $S$, and let $E_1,\dots,E_\ell$ be the corresponding eigenspaces. For all $V \in \mathfrak{g}_2$, we have $J_V^2 = - \|V\|^2 S^2$. Thus $J_V$ commutes with $S^2$, hence with $S$. We can then split the skew-symmetric operator $J_V$ in the orthogonal direct sum
\begin{equation}
J_V = \bigoplus_{i=1}^\ell \lambda_i J_V^i, \qquad J_V^i : = \frac{1}{\lambda_i} J_V|_{E_i}.
\end{equation}
The generalized H-type condition \eqref{eq:13} is equivalent to the fact that each linear map $\mathfrak{g}_2 \ni V \mapsto J_V^i$ must satisfy a standard H-type condition in dimension $k_i = \dim E_i$. By the classical results mentioned above, this is possible if and only if for all $i=1,\dots,\ell$, the condition $n-k < \rho(k_i)$ is satisfied. Hence there exists a generalized H-type Carnot group of dimension $n$, rank $k$, with prescribed eigenvalues $\lambda_i$ and multiplicities $k_i$ if and only if
\begin{equation}
n-k < \min\{\rho(k_i)\mid i=1,\dots,\ell\}.
\end{equation}
Notice that the dimension of the kernel of $S$ does not play any role in this condition. As a consequence, generalized H-type Carnot groups with odd rank do exist.

\subsection{Coordinate presentation} \label{s:normalforms}
 
The group exponential map $\mathrm{exp}_{G} : \mathfrak{g} \to G$ associates with $V \in \mathfrak{g}$ the element $\gamma_V(1)$, where $\gamma_V: [0,1] \to G$ is the unique integral curve of the vector field defined by $V$ such that $\gamma_V(0) = 0$. Since $G$ is simply connected and $\mathfrak{g}$ is nilpotent, $\mathrm{exp}_G$ is a smooth diffeomorphism.
The choice of an orthonormal basis $e_1,\dots,e_k$ on $\mathfrak{g}_1$ and $f_1,\dots,f_{n-k}$ on $\mathfrak{g}_2$ yields, through the group exponential map, analytic coordinates $(x,z) \in \R^k \times \R^{n-k}$ on $G$ such that $p =(x,z)$ if and only if
\begin{equation}
p = \mathrm{exp}_{G} \left(\sum_{i=1}^{k} x_i e_i + \sum_{\alpha=1}^{n-k} z_\alpha f_\alpha\right).
\end{equation}
Since $[\mathfrak{g}_1,\mathfrak{g}_1] = \mathfrak{g}_2$, there exists a linearly independent family of $k\times k$ skew-symmetric matrices $L^\alpha$, for $\alpha=1,\dots,n-k$, such that the Lie product on $\mathfrak{g}$ reads
\begin{equation}
[e_i,e_j]_{\mathfrak{g}} = \sum_{\alpha=1}^{n-k} L_{ij}^\alpha f_\alpha, \qquad i,j=1,\dots,k.
\end{equation}
In these coordinates, the sub-Riemannian distribution on $G$ is generated by the following set of global orthonormal left-invariant vector fields
\begin{equation}\label{eq:fields}
X_i = \partial_{x_i} - \frac{1}{2}\sum_{\alpha=1}^{n-k}\sum_{j=1}^{k} L_{ij}^\alpha x_j \partial_{z_\alpha}, \qquad i =1,\dots, k,
\end{equation}
obtained by left translation of the corresponding $e_i$. Furthermore, $\partial_{z_\alpha}$ are left-invariant vector fields obtained by left translation of the basis $f_\alpha$ of $\mathfrak{g}_2$.

Observe in particular that
\begin{equation}
[X_i,X_j] = \sum_{\alpha=1}^{n-k}L^{\alpha}_{ij} \partial_{z_\alpha}, \qquad  i,j = 1,\dots,k.
\end{equation}
If $V = \sum v_\alpha f_\alpha$, then $J_V$ is represented by the skew-symmetric matrix $L_v:= \sum v_\alpha L^\alpha$. Notice that $v \mapsto L_v$ is linear. Property \eqref{eq:13} corresponds to the following relations:
\begin{equation}
L_v L_w + L_w L_v = -2 (v \cdot w) S^{2}, \qquad \forall v,w \in \R^{n-k},
\end{equation}
where the dot denotes the Euclidean scalar product. In particular, $L_v$ and $L_w$ anti-commute when $v \perp w$, and $L_v^2= -|v|^2 S^2$. By a suitable orthogonal transformation of the $e_i$'s, we assume that $S$ is diagonal. Since $L_v^2 = - |v|^2 S^2$ and $L_v$ is skew-symmetric the non-zero eigenvalues of $S$ appear in pairs and, since $S \geq 0$, we have
\begin{equation}\label{eq:normalformS}
S = \begin{pmatrix}
\mathbbold{0}_{k-2d} & & & \\
& \alpha_1 \mathbbold{1} & & \\
& & \ddots & \\
& & & \alpha_d \mathbbold{1}
\end{pmatrix}, \qquad \mathbbold{1} = \begin{pmatrix}
1 & 0 \\ 0 & 1
\end{pmatrix},
\end{equation}
for some $1\leq d \leq k$, and with $0<\alpha_1\leq \alpha_2\leq \dots\leq \alpha_d$. As a consequence, for all $v \in \R^{n-k}$, the real normal form $\tilde{L}_v$ of $L_v$ is 
\begin{equation}\label{eq:normalformlLv}
\tilde{L}_v = \begin{pmatrix}
\mathbbold{0}_{k-2d} & & & \\
& |v|\alpha_1 \mathbbold{J} & & \\
& & \ddots & \\
& & & |v|\alpha_d \mathbbold{J}
\end{pmatrix}, \qquad \mathbbold{J} = \begin{pmatrix}
0 & 1 \\
-1 & 0
\end{pmatrix},
\end{equation}
and its spectrum is $\sigma(L_v) = \{0,\pm i |v|\alpha_1,\dots,\pm i |v|\alpha_d \}$, where the zero value is not present if $k=2d$ (that is, when $S>0$).

Henceforth, we fix a choice of exponential coordinates $(x,z)$, in such a way that the sub-Riemannian structure on $G$ is defined by the left-invariant fields \eqref{eq:fields}, and \eqref{eq:normalformS} holds.

\subsection{Lie group isomorphisms vs sub-Riemannian isometries}
Notice that all generalized H-type groups of fixed rank $k$ and dimension $n$ are isomorphic as Lie groups, but not isometric as sub-Riemannian structures.

Let $G$ and $G'$ be generalized H-type groups and $\phi: G\to G'$ be a Lie group isomorphism. In particular $\phi_{*}:\mathfrak{g}\to \mathfrak{g}'$ is a Lie algebra isomorphism preserving the grading and we denote by $M:\mathfrak{g}_{1}\to \mathfrak{g}_{1}'$ and $T  :\mathfrak{g}_{2}\to \mathfrak{g}_{2}'$ its restrictions. If $\phi$ is also a sub-Riemannian isometry (and thus $M:\mathfrak{g}_{1}\to \mathfrak{g}_{1}'$ is an isometry of inner product spaces), then
\begin{equation}
S'=\frac{\Tr(TT^{*})}{n-k}MSM^*,
\end{equation}
where $A^{*}$ denotes the adjoint of a linear operator $A$ with respect to the scalar products on the corresponding spaces.

It follows that the ratios between the non-zero eigenvalues $\alpha_{i}$, with $i=1,\ldots,d$, of the symmetric operator $S$ and the dimensions of the corresponding eigenspaces are preserved by isometries and distinguish between non-isometric sub-Riemannian structures.

This can be seen already in the subclass of contact Carnot groups. In this case the structure of the cut locus (see \cite{BBN16}) and the isometry group of the structure (see \cite{LR17}) highly depend on the ratio $\alpha_{1}/\alpha_{2}$.

\section{Measure contraction property and geodesic dimension}\label{s:mms}

Let $(X,d,\mu)$ be a metric measure space, where $(X,d)$ is a length space and $\mu$ be a Borel measure such that $0< \mu(\ball(x,r)) < +\infty$ for any metric ball $\ball(x,r)$. 
Moreover let us assume that $(X,d)$ has negligible cut loci: for any $x \in X$ there exists a negligible set $\mathcal{C}(x)$ and a measurable map $\Phi^x : X \setminus \mathcal{C}(x) \times [0,1] \to X$, such that the curve $\gamma(t)=\Phi^x(y,t)$ is the unique minimizing geodesic from $x$ with $y$.

For any set $\Omega$, we consider its geodesic homothety of center $x \in X$ and ratio $t \in [0,1]$:
\begin{equation}\label{eq:omot}
\Omega_t:= \{\Phi^x(y,t) \mid y \in \Omega \setminus \mathcal{C}(x)\}.
\end{equation}
For any $K \in \R$, define the function
\begin{equation}\label{eq:sturm}
s_K(t):= \begin{cases}
(1/\sqrt{K}) \sin(\sqrt{K} t) & \text{if } K>0, \\
t & \text{if } K=0, \\
(1/\sqrt{-K}) \sinh(\sqrt{-K} t) & \text{if } K<0.
\end{cases}
\end{equation}
\begin{definition}[Ohta \cite{Ohta-MCP}]\label{d:MCP}
Let $K \in \R$ and $N>1$, or $K \leq 0$ and $N=1$. We say that $(X,d,\mu)$ satisfies the \emph{measure contraction property} $\mathrm{MCP}(K,N)$ if for any $x \in M$ and any measurable set $\Omega$ with  with $0< \mu(\Omega)< + \infty$ (and with $\Omega \subset \ball(x,\pi\sqrt{N-1/K})$ if $K > 0$)
\begin{equation}\label{eq:MCP}
\mu(\Omega_{t}) \geq \int_{\Omega} t \left[\frac{s_K(t d(x,z)/\sqrt{N-1})}{s_K(d(x,z)/\sqrt{N-1})}\right]^{N-1} d\mu(z) , \qquad \forall t \in [0,1],
\end{equation}
where we set $0/0 = 1$ and the term in square bracket is $1$ if $K \leq 0$ and $N=1$.
\end{definition}
Ohta defines the $\mathrm{MCP}$ for general length spaces, possibly with non-negligible cut loci. In our setting, Definition~\ref{d:MCP} is equivalent to Ohta's, see \cite[Lemma 2.3]{Ohta-MCP}.

\subsection{The geodesic dimension}

The geodesic dimension was introduced in \cite{ABR-Curvature} for sub-Riemannian structures and extended in \cite{MCPc1} to the more general setting of metric measure spaces (with negligible cut loci).
\begin{definition} \label{d:gd}
Let $(X,d,\mu)$ be a metric measure space with negligible cut locus. For any $x \in X$ and $s > 0$, define
\begin{equation}\label{eq:criticalratio}
C_s(x):=\sup\left\lbrace \limsup_{t \to 0^+} \frac{1}{t^s}\frac{\mu(\Omega_t)}{\mu(\Omega)} \mid \Omega \text{ measurable, bounded, $0<\mu(\Omega)<+\infty$}\right\rbrace,
\end{equation}
where $\Omega_t$ is the homothety of $\Omega$ with center $x$ and ratio $t$ as in \eqref{eq:omot}. The \emph{geodesic dimension} of $(X,d,\mu)$ at $x \in X$ is the non-negative real number
\begin{equation}\label{eq:gddefs}
\mathcal{N}(x)  := \inf \{s > 0 \mid C_s(x) = +\infty \} = \sup \{s > 0 \mid C_s(x) = 0 \},
\end{equation}
with the conventions $\inf \emptyset = + \infty$ and $\sup \emptyset = 0$.
\end{definition}
Roughly speaking, the measure $\mu(\Omega_t)$ vanishes at least as $t^{\mathcal{N}(x)}$ or more rapidly, for $t \to 0$. The two definitions in \eqref{eq:gddefs} are equivalent since $s \geq s'$ implies $C_s(x) \geq C_{s'}(x)$. 

When $(X,d,\mu)$ is a metric measure space defined by an equiregular sub-Rieman\-nian or Riemannian structure, equipped with a smooth measure $\mu$, we have
\begin{equation}
\mathcal{N}(x) \geq \dim_H(X) \geq \dim(X), \qquad \forall x \in X,
\end{equation}
and both equalities hold if and only if $(X,d,\mu)$ is Riemannian (see \cite[Proposition 5.49]{ABR-Curvature} and \cite{MCPc1}). Here $\dim_H(X)$ denotes the Hausdorff dimension of the metric space $(X,d)$. We refer to \cite{MCPc1} for a discussion of the other properties of the geodesic dimension. We only remark that, in the case of left-invariant structure on Carnot (or, more generally, Lie) groups, thanks to left-invariance, the number $\mathcal{N}(x)$ is independent of $x\in X$, and we denote it by $\mathcal{N}$. As a consequence of the explicit formula for $\mathcal{N}$ on sub-Riemannian structures \cite[Section 7]{MCPc1}, we have the following.
\begin{lemma}\label{l:fact}
The geodesic dimension of generalized H-type Carnot groups is $k + 3(n-k)$.
\end{lemma}

\begin{rmk} 
For simplicity, and since all step $2$ sub-Riemannian structures satisfy this assumption, we defined the geodesic dimension and the measure contraction property for metric measure spaces with negligible cut loci. Nevertheless, these notions can be generalized to any metric measure space with little effort. Finally, let us recall that it is not known whether general Carnot groups have negligible cut loci (this is related with the Sard conjecture in sub-Riemannian geometry \cite{AAAopenproblems,RT-MorseSard,DMOPV-Sard}).
\end{rmk}

%% file: proofs.tex
\section{Geodesics of generalized H-type Carnot groups} \label{s:five}

In this section we compute the explicit expression of geodesics in generalized H-type Carnot groups. By left-invariance, it is sufficient to consider geodesics starting from the identity $e = (0,0)$. We denote covectors $\lambda = \sum_{i=1}^k u_i dx_i + \sum_{\alpha=1}^{n-k} v_\alpha dz_\alpha \in T^*_eG$, with $(u,v) \in \R^k \times \R^{n-k}$ coordinates on $T^*_e G$ dual to the exponential coordinates $(x,z) \in \R^k \times \R^{n-k}$ of $G$ (see Section~\ref{s:normalforms}).

\SkipTocEntry \subsection*{Convention} For any analytic function $f: \mathbb{C} \to \mathbb{C}$ and any square matrix $M$, the expression $f(M)$ denotes the matrix valued function defined by the Taylor series of $f$ at zero. For example,
\begin{equation}
\frac{\sin M}{M} := \sum_{\ell=0}^{\infty} \frac{M^{2\ell}}{(2\ell+1)!},
\end{equation}
which is well defined for any $M$, possibly degenerate. With this convention, we routinely simplify matrix computations by reducing to functional calculus, and then evaluating on matrices. For example, for any pair $f,g$ of analytic functions, $f(M) g(M) = (fg)(M)$. 
%{\red Furthermore, if $M=M(v)$ depends smoothly on a parameter $v$, then
%\begin{equation}
%\partial_v f(M(v)) = f'(M(v)) \partial_v M(v),
%\end{equation}
%and similarly for integration.} 
Furthermore, for $s \in \R$ we observe that
\begin{equation}\label{eq:derivatio}
\partial_s (f(s M)) = f'(s M) M,
\end{equation}
and similarly for integration. If $M$ can be diagonalized and has eigenvalues $\lambda_i \in \mathbb{C}$, then
\begin{equation}
\det(f(M)) = \prod_i f(\lambda_i).
\end{equation}

\begin{lemma}[Normal geodesics]\label{l:geods}
The exponential map $\exp_e :T^*_e G \to G$ of a generalized H-type Carnot group is given by $\exp_e(u,v) = (x,z)$, where
\begin{align}
x & = f(L_v) u, \\
z & = \left(\frac{g(L_v) u \cdot u }{2|v|}\right) \frac{v}{|v|},
\end{align}
where $|\cdot|$ is the Euclidean norm and $f,g$ are the analytic functions
\begin{equation}\label{eq:fg}
f(z) = \frac{1-e^{-z}}{z}, \qquad g(z) = 1 -\frac{\sinh(z)}{z}.
\end{equation}
Geodesics $\gamma_\lambda:[0,1] \to G$ with initial covector $\lambda =(u,v)$ are given by $\gamma_\lambda(t) = \exp_e(tu,tv)$.
\end{lemma}
\begin{rmk}\label{eq:geodesicsclassical}
For classical H-type groups, $L_v^2 = -|v|^2 \mathbbold{1}$, and the above expressions yield
\begin{align}
x & = \left(\frac{\sin |v|}{|v|}\mathbbold{1} + \frac{\cos |v|-1}{|v|^2}L_{v} \right) u ,\\
z  & = |u|^2\left(\frac{|v|-\sin |v|}{2|v|^2}\right)\frac{v}{|v|}.
\end{align}
If $v= 0$, our convention implies $\exp_e(u,0) = (u,0)$.
\end{rmk}
\begin{proof}
Let $h_x = (h_x^1,\dots,h_x^k) : T^*G \to \R^k$ and $h_z = (h_z^1,\dots,h_{z}^{n-k}) : T^*G \to \R^{n-k}$, where $h_x^{i}(\lambda):= \langle \lambda, X_i\rangle$, for $i=1,\dots,k$ and $h_z^\alpha(\lambda) := \langle \lambda, \partial_{z_\alpha}\rangle$ for all $\alpha =1,\dots,n-k$. Thus, $H = \frac{1}{2}|h_x|^2$. Hamilton equations are
\begin{equation}
\dot{h}_z = 0, \qquad
\dot{h}_x = -L_{h_z} h_x, \qquad \dot{x} = h_x, \qquad \dot{z}^\alpha = -\tfrac{1}{2}h_x \cdot L^\alpha x,
\end{equation}
the dot denoting the derivative with respect to $t$. Set $h_z(0)=v$ and $h_x(0) =u$. Then,
\begin{equation}
h_z(t) = v, \qquad h_x(t) = e^{-t L_v} u. 
%= \left(\cos(t |v| S)  - \frac{\sin(t |v| S)}{|v|S}L_v\right) u.
\end{equation}
The equations for $(x,z)$ can be easily integrated. Exploiting our convention, we obtain
\begin{equation}
x(t) = \int_0^t e^{-s L_v} u \, ds = \frac{\mathbbold{1} - e^{-t L_v}}{L_v} u = f(t L_v) tu.
%\left(\frac{\sin(t|v|S)}{|v| S}  + \frac{\cos(t|v| S)-1}{(|v|S)^2}L_v\right) u.
\end{equation}
To deal with the $z$ coordinate, observe that for any $w\perp v \in \R^{n-k}$, we have
\begin{equation}
2w \cdot \dot{z} = - h_x \cdot L_w x = \sum_{m,n=0}^{\infty} c_{m,n} t^{n+m} (L_v^m u)\cdot (L_w L_v^n u),
\end{equation}
for some coefficients $c_{m,n}$. We claim that each term $(L_v^m u)\cdot (L_w L_v^n u) = 0$, and thus the above expression vanishes identically. To prove this claim, assume without loss of generality that $m\geq n$. The argument is different depending on whether the difference $m-n$ is even or odd. If $m=n+2\ell$, we have
\begin{equation}
(L_v^{n+2\ell} u) \cdot (L_w L_v^n u) = (-1)^\ell (L_v^{n+\ell} u) \cdot (L_v^\ell L_w L_v^n u) = (L_v^{n+\ell}u) \cdot (L_w L_v^{n+\ell}u) = 0.
\end{equation}
If $m = n+2\ell +1$, on the other hand, we have
\begin{multline}
(L_v^{n+2\ell+1}u)\cdot(L_w L_v^n u) = -(L_v^n u) \cdot (L_v^{2\ell+1} L_w L_v^n u) = (L_v^{n} u)\cdot (L_w L_v^{n+2\ell+1} u) = \\
= -(L_w L_v^{n} u) \cdot (L_v^{n+2\ell+1}u) = -(L_v^{n+2\ell+1}u) \cdot (L_w L_v^{n} u),
\end{multline}
which implies $(L_v^{n+2\ell+1}u)\cdot(L_w L_v^n u)= 0$, and proves the claim.

As a consequence of the claim, $z(t)$ is parallel to $v$. Then, we have
\begin{align}
\dot{z}(t) & = \frac{\dot{z}(t)\cdot v}{|v|}\frac{v}{|v|} = -\frac{h_x(t) \cdot L_v x(t) }{2|v|}\frac{v}{|v|} \\
& = -\frac{t L_v f(tL_v)e^{t L_v}u \cdot u}{2|v|} \frac{v}{|v|} \\
& = \frac{[\mathbbold{1}-\cosh(t L_v)]u \cdot u}{2|v|} \frac{v}{|v|},
\end{align}
where in the second line we used that $L_v$ is skew-symmetric, and in the third line we replaced the matrix $tL_v f(tL_v)e^{t L_v}$ with its symmetric part, which can be easily computed from the explicit form of $f$. Integration on the time interval $[0,1]$ concludes the proof.
%Thus, to conclude the proof, it is sufficient to compute the following integral, and then replace $z\mapsto L_v$:
%\begin{equation}
%\int_0^1 \left(tz f(-tz)e^{-tz} -tz f(tz)e^{tz} \right)dt  = 1-\frac{\sinh(z)}{z}. \qedhere
%\end{equation}
\end{proof}

Consider a non-zero covector $\lambda = (u,v)$, with $u \in \ker S$. Since $Su=0$ implies $L_v u = 0$, by Lemma~\ref{l:geods}, these curves are straight lines of the form
\begin{equation}
\exp_e(t u,tv) = (tu,0), \qquad \forall v \in \R.
\end{equation}
In particular, there is an infinite number of initial covectors yielding the same geodesic, which is thus abnormal. All abnormal geodesics are of this type.

\begin{lemma}[Abnormal curves and abnormal geodesics]\label{l:abnormals} Abnormal curves of generalized H-type groups are the Lipschitz curves $\gamma_u(t):[0,1] \to G$ with control $u \in L^\infty([0,1],\R^k)$ such that $u(t) \in \ker S$. The curve $\gamma_u(t)$ is an abnormal geodesic if and only if it is a normal geodesic with initial covector $\lambda = (u,v)$ with $u \in \ker S$, i.e.
\begin{equation}
\exp_e(t u,tv) = (tu,0), \qquad \forall v \in \R.
\end{equation}
\end{lemma}
\begin{proof}
We use the explicit characterization of abnormal curves of Theorem~\ref{t:utile}. Let $h_x = (h_x^1,\dots,h_x^k) : T^*G \to \R^k$ and $h_z = (h_z^1,\dots,h_{z}^{n-k}) : T^*G \to \R^{n-k}$, where $h_x^{i}(\lambda):= \langle \lambda, X_i\rangle$, for $i=1,\dots,k$ and $h_z^\alpha(\lambda) := \langle \lambda, \partial_{z_\alpha}\rangle$ for all $\alpha =1,\dots,n-k$. In these coordinates,
\begin{equation}
\distr^\perp = \{(x,z,h_x,h_z) \in T^*G  \mid h_x = 0\}.
\end{equation}
Let $\gamma_u(t)$ be an abnormal geodesic, with $u \in L^\infty([0,1],\R^k)$. Let $\lambda(t) = (x(t),z(t),0,h_z(t))$ be a Lipschitz curve in $\distr^\perp$ such that $\pi(\lambda(t)) = \gamma_u(t)$. In particular
\begin{equation}
\dot{\lambda}(t) = \sum_{i=1}^k u_i(t) X_i + \sum_{\alpha = 1}^{n-k} \dot{h}_z^\alpha(t) \partial_{h_{z}^\alpha},
\end{equation}
where the vector field $X_i$ in the r.h.s.\ is regarded as a vector field on $T^*G$. We now demand that $\sigma_{\lambda(t)}(\dot\lambda,T_{\lambda}\distr^\perp) = 0$. First, since $\partial_{z_\alpha}|_{\lambda(t)} \in T_{\lambda(t)}\distr^\perp$, we obtain the condition
\begin{equation}
0 = \sigma_{\lambda(t)}(\dot\lambda,\partial_{z_\alpha}) = \dot{h}_{z}^\alpha(t), \qquad \forall \alpha = 1,\dots,n-k,\quad \text{a.e. } t \in [0,1].
\end{equation}
Hence $h_z(t)$ is constant. This constant must be $h_z(t) = v \neq 0$, otherwise the associated abnormal Lagrange multiplier $\lambda_1 = \lambda(1) \in \distr^\perp_{(x,z)} \subset T_{(x,z)}^*G$ would be the zero covector, which is impossible by Theorem~\ref{t:utile}. Furthermore, since $X_j|_{\lambda(t)} \in T_{\lambda(t)}\distr^\perp$, we have
\begin{equation}
0 = \sigma_{\lambda(t)}(\dot\lambda,X_j) = \sum_{i=1}^k \sum_{\alpha=1}^{n-k} u_i(t) L_{ij}^\alpha v_\alpha, \qquad \forall j=1,\dots,k,\quad \text{a.e. } t \in [0,1],
\end{equation}
which implies $u(t) \in \ker L_v$. Since $v \neq 0$ and $L_v^2 = -|v|^2 S^2$, this condition is equivalent to $u(t) \in \ker S$. Thus, an horizontal curve $\gamma_u(t)$ is abnormal if and only if it satisfies
\begin{equation}
\dot{\gamma}_u(t) = \sum_{i=1}^k u_i(t) X_i(\gamma_{u}(t)), \qquad u(t) \in \ker S, \qquad u \in L^\infty([0,1],\R^k).
\end{equation}
Using the explicit form of the vector fields $X_i$, this implies $x(t) \in \ker S$ and, in turn, $z(t) = 0$. In particular $\gamma_u(t) = \int_0^t u(s) ds$ lies in the plane $\R^k \subset \R^k \oplus \R^{n-k}$ of $G$, and its length (resp.\ energy) is equal to its Euclidean length (resp.\ energy) in $\R^k$. If we demand $\gamma_u(t)$ to be a geodesic, i.e.\ a locally energy minimizing curve, we have that, $x(t)$ must be a straight line of the form $x(t) = tu$, in particular $u(t) = u \in \ker S$. From Lemma~\ref{l:geods}, we observe that the abnormal geodesics $x(t)=(tu,0)$, with $u \in \ker S$, are also normal, with initial covectors $\lambda = (u,v)$, for any $v \in \R^{n-k}$.
\end{proof}

We conclude this section by observing that all generalized H-type groups are the metric product of an ideal one (i.e.\ with no non-trivial abnormal minimizers) and a copy of the Euclidean space of the appropriate dimension.
\begin{prop}
For any generalized H-type group $G$ there exists a unique (up to isometries) ideal generalized H-type group $\hat{G}$ such that $G = \mathbb{R}^{k-2d} \times \hat{G}$, with $k-2d = \dim\ker S$.
\end{prop}
\begin{proof}
In the coordinate presentation of Section~\ref{s:normalforms}, we identify $G = \R^k \times \R^{n-k}$. According to \eqref{eq:normalformS}, we further split $G = \R^{k-2d} \times \R^{2d} \times \R^{n-k}$, where the first $\R^{k-2d}$ component represents the kernel of $S$. The sub-Riemannian structure induced by the projection on the component $\R^{k-2d}$ is the Euclidean one. On the other hand, the sub-Riemannian structure induced by the projection on the component $\R^{2d} \times \R^{n-k}$ is the one of a generalized H-type group of rank $2d$, with non-degenerate operator $\hat{S}$ (given by the restriction of $S$ to the orthogonal complement of $\ker S$), which we call $\hat{G}$. By Lemma~\ref{l:abnormals}, $\hat{G}$ is ideal. 

By Lemma~\ref{l:geods}, all geodesics of $G$ are straight lines in the first $\R^{k-2d}$ component (representing the kernel of $S$), and are geodesics for $\hat{G}$ in the remaining $\R^{2d} \times \R^{n-k}$ component. This yields that $G$ is the metric product $\R^{k-2d} \times \hat{G}$. The uniqueness follows easily.
\end{proof}

\section{Jacobian determinant and cotangent injectivity domain}\label{s:proof}

The most important steps in the proof of Theorem~\ref{t:htype} are the computation of the Jacobian determinant of the exponential map, and the characterization of the cotangent injectivity domain of generalized H-type Carnot groups. These results have independent interests. For example, the analogous formulas for the Heisenberg group and, more generally, corank 1 Carnot groups have been employed in \cite{BKS16,BKS17} to obtain Jacobian determinant inequalities.

\SkipTocEntry \subsection*{Notation} 
For the statement of the next Lemma \ref{l:jacobian}, we recall the notation of Section \ref{s:normalforms}. For each $v \in \R^{n-k}$, the skew-symmetric matrix $L_v$ has real normal form given by \eqref{eq:normalformlLv}, with non-zero singular values $\alpha_i |v|$, with $i=1,\dots,d$, with $0<\alpha_1\leq \dots \leq \alpha_d$, repeated according to their multiplicity. To each $\alpha_j$ corresponds a pair of non-zero eigenvalues of the symmetric matrix $S$ defining the H-type group, see \eqref{eq:normalformS}. If $v \neq 0$, after applying an orthogonal transformation putting $L_v$ in its real normal form \eqref{eq:normalformlLv}, we denote $u=(u_0,u_1,\dots,u_d)$, where $u_0 \in \R^{k-2d}$ is the component in the kernel of $L_v$, while each $u_i \in \R^2$, for $i=1,\dots,d$, is the component of $u$ in the real eigenspace associated with the singular value $\alpha_i$. Notice that this decomposition is non-unique. The forthcoming lemma can be expressed in a more intrinsic way by introducing multiplicities of singular values, but for further manipulations we prefer to present it in its actual form.

\begin{lemma}[Jacobian determinant]\label{l:jacobian}
The Jacobian determinant of the exponential map $\exp_e(u,v) = (x,z)$ of a generalized H-type group is
\begin{multline}
J(u,v)= \frac{2^{2d}}{\alpha^2|v|^{2d+2}}\left(\sum_{i=1}^d |u_i|^2 \frac{\alpha_i|v| - \sin(\alpha_i|v|)}{2|v|^3 \alpha_i}\right)^{n-k-1} \times \\
\times \sum_{i=1}^d |u_i|^2 \prod_{j \neq i}\sin\left(\frac{\alpha_j |v|}{2}\right)^2 \sin\left(\frac{\alpha_i |v|}{2}\right)\left(\sin\left(\frac{ \alpha_j |v|}{2}\right) - \frac{\alpha_j |v|}{2} \cos\left(\frac{ \alpha_j |v|}{2}\right)\right),
\end{multline}
where $\alpha = \prod_{j=1}^d \alpha_j$ is the product of the non-zero singular values of $L_v/|v|$. If $v = 0$, the formula must be taken in the limit $v \to 0$. In particular,
\begin{equation}
J(u,0)= \left(\sum_{i=1}^d \frac{(\alpha_i |u_i|)^2}{12}\right)^{n-k} = \left(\frac{|Su|^2}{12}\right)^{n-k}.
\end{equation}
\end{lemma}
\begin{rmk}
For H-type groups, $L_v^2 = -|v|^2 \mathbbold{1}$, and the above expression reduces to
\begin{equation}
J(u,v) =  \frac{2^k|u|^{2(n-k)}}{|v|^{k+2}} \left(\frac{|v|-\sin|v|}{2|v|^3}\right)^{n-k-1} \sin\left(\frac{|v|}{2}\right)^{k-1}\left(\sin\left(\frac{|v|}{2}\right) - \frac{|v|}{2}\cos\left(\frac{|v|}{2}\right)\right).
\end{equation}
\end{rmk}
\begin{proof}
It is convenient to rewrite the exponential map $\exp_e(u,v) = (x,z)$ using polar coordinates on the $v$ and $z$ components. Namely, $v \mapsto (|v|,\hat{v} = v/|v|) \in \R_+ \times \mathbb{S}^{n-k-1}$, and $z \mapsto (|z|,\hat{z} = z/|z|) \in \R_+ \times \mathbb{S}^{n-k-1}$. One can check that $g(L_v)u \cdot u \geq 0$. Lemma~\ref{l:geods} reads
\begin{equation}\label{eq:polargeod}
x = f(L_v) u, \qquad |z| = \frac{g(L_v)u \cdot u}{2|v|}, \qquad \hat{z} = \hat{v}.
\end{equation}
Thus, passing in polar coordinates, we have
\begin{equation}\label{eq:polartocart}
J(u,v) = \frac{|z|^{n-k-1}}{|v|^{n-k-1}} \tilde{J}(u,|v|,\hat{v}),
\end{equation}
where $\tilde{J}$ is the Jacobian determinant in polar coordinates,
\begin{equation}\label{eq:protojac}
\tilde{J}(u,|v|,\hat{v}) = \det \begin{pmatrix}
\frac{\partial x}{\partial u} & \frac{\partial x}{\partial |v|} & \frac{\partial x}{\partial \hat{v}} \\
\frac{\partial |z|}{\partial u} & \frac{\partial |z|}{\partial |v|} & \frac{\partial |z|}{\partial \hat{v}} \\
\frac{\partial \hat{z}}{\partial u} & \frac{\partial \hat{z}}{\partial |v|} & \frac{\partial \hat{z}}{\partial \hat{v}} \\
\end{pmatrix} = \det
\begin{pmatrix}
\frac{\partial x}{\partial u} & \frac{\partial x}{\partial |v|} & \frac{\partial x}{\partial \hat{v}} \\
\frac{\partial |z|}{\partial u} & \frac{\partial |z|}{\partial |v|} & \frac{\partial |z|}{\partial \hat{v}} \\
\mathbbold{0} & \mathbbold{0} & \mathbbold{1} \\
\end{pmatrix}=
\det \begin{pmatrix}
\frac{\partial x}{\partial u} & \frac{\partial x}{\partial |v|} \\
\frac{\partial |z|}{\partial u} & \frac{\partial |z|}{\partial |v|}
\end{pmatrix}.
\end{equation}
Using the explicit expressions \eqref{eq:polargeod}, recalling that $L_v = |v| L_{\hat{v}}$, we obtain
\begin{align}
\frac{\partial x}{\partial u} & = f(L_v), & \frac{\partial x}{\partial |v|} & = \frac{f'(L_v) L_v}{|v|} u, \\
\frac{\partial |z|}{\partial u} & = \frac{u^*g(L_v)}{|v|}, & \frac{\partial |z|}{\partial |v|} & = \frac{[g'(L_v)L_v - g(L_v)]u \cdot u}{2|v|^2},
\end{align}
where in the third equation we used the fact that $g(L_v)^* = g(L_v^*) = g(-L_v) = g(L_v)$, the prime denotes the derivative of $f$ and $g$ defined in \eqref{eq:fg}, and where to compute the derivatives w.r.t.\ $|v|$ we exploited property \eqref{eq:derivatio}.

The matrix on the right hand side of \eqref{eq:protojac} has the following block structure
\begin{equation}
\Xi = \begin{pmatrix}
A & B \\
C & D
\end{pmatrix},
\end{equation}
with $A = f(L_v)$ a $k\times k$ matrix. If $A$ is invertible, $\det(\Xi) = \det(A)\det(D - CA^{-1} B)$. If $A$ is not invertible, we can still apply this formula replacing $A$ with a sequence of invertible matrices $A_n \to A$, and then taking the limit. We omit this standard density argument, and in the following, we will formally manipulate all the expression as if $A$ were invertible. Thus, using the fact that $f(L_v)^{-1} = (1/f)(L_v)$, we obtain
\begin{equation}
\tilde{J}(u,|v|,\hat{v}) = \frac{ \det(f(L_v))}{2|v|^2}  u^*\underbrace{\left[g'(L_v)L_v - g(L_v) - \frac{2g(L_v)f'(L_v) L_v}{f(L_v)}\right]}_{\Theta(L_v)}u.
\end{equation}
To compute the non trivial matrix $\Theta(L_v)$, we first notice that we can replace it with its symmetric part $(\Theta(L_v) + \Theta(L_v^*))/2$ without changing the result. Moreover, $\Theta(L_v)^* = \Theta(L_v^*) = \Theta(-L_v)$. Exploiting our convention, it is sufficient to perform this computation for the corresponding analytic function of one variable $\Theta(z)$. We obtain
\begin{equation}
\frac{\Theta(z) + \Theta(-z)}{2} = 2\frac{\sinh\left(\frac{z}{2}\right) - \frac{z}{2}\cosh\left(\frac{z}{2}\right)}{\sinh\left(\frac{z}{2}\right)}.
\end{equation}
Thanks to this expression, we obtain
\begin{equation}
\tilde{J}(u,|v|,\hat{v})=\frac{\det(f(L_v))}{|v|^2}u^*\left[\frac{\sinh\left(\frac{L_v}{2}\right) - \frac{L_v}{2}\cosh\left(\frac{L_v}{2}\right)}{\sinh\left(\frac{L_v}{2}\right)}\right]  u.
\end{equation}
We now perform an orthogonal change of basis in the $u$ space, in such a way that $L_v$ is in normal form \eqref{eq:normalformlLv}. Furthermore, we decompose $u=(u_0,u_1,\dots,u_d)$ according to the real eigenspaces corresponding to the kernel of $L_v$ and the pairs of eigenvalues $\pm i |v| \alpha_j$, for $j=1,\dots,d$ of $L_v$. In particular, $u_0 \in \R^{k-2d}$, and $u_i \in \R^{2}$. Then, we get
\begin{align}
\tilde{J}(u,|v|,\hat{v}) & =\frac{f(0)^{k-2d}}{|v|^2} \prod_{j=1}^d f(i\alpha_j|v|)f(-i\alpha_j|v|)
\sum_{i=1}^d |u_i|^2 \frac{\sin\left(\frac{\alpha_i|v| }{2}\right) - \frac{\alpha_i|v|}{2} \cos\left(\frac{\alpha_i|v| }{2}\right)}{\sin\left(\frac{ \alpha_i|v|}{2}\right)} \\
& = \frac{1}{|v|^2}\prod_{j=1}^d \left(\frac{2}{\alpha_j |v|} \sin\left(\frac{\alpha_j |v|}{2}\right)\right)^2
\sum_{i=1}^d |u_i|^2 \frac{\sin\left(\frac{\alpha_i |v|}{2}\right) - \frac{\alpha_i |v|}{2} \cos\left(\frac{\alpha_i |v|}{2}\right)}{\sin\left(\frac{\alpha_i |v|}{2}\right)} \\
& = \frac{2^{2d}}{\alpha^2|v|^{2d+2}}\sum_{i=1}^d |u_i|^2 \prod_{j \neq i}\sin\left(\frac{\alpha_j |v|}{2}\right)^2 \sin\left(\frac{\alpha_i |v|}{2}\right)\times \\
& \qquad\qquad\qquad\qquad\qquad\hfill\qquad\qquad \times \left(\sin\left(\frac{\alpha_j |v| }{2}\right) - \frac{\alpha_j |v|}{2} \cos\left(\frac{\alpha_j|v| }{2}\right)\right),
\end{align}
where in the second line we used the fact that $f(0) = 1$, and $\alpha = \prod_{i=1}^d \alpha_i$. From this expression, we recover the expression for the Jacobian determinant in the original Cartesian coordinates $(u,v)$ through \eqref{eq:polartocart}, and observing that
\begin{equation}
|z| = \frac{g(L_v) u \cdot u}{2|v|} = \sum_{i=1}^d |u_i|^2 \frac{\alpha_i |v| - \sin(\alpha_i |v|)}{2 \alpha_i|v|^2}. \qedhere
\end{equation}
\end{proof}

\begin{lemma}[Cotangent injectivity domain]\label{l:ingjd}
Let $\alpha_{d}$ be the largest eigenvalue of $S$, and 
\begin{equation}
D := \left\lbrace \lambda = (u,v) \in T_e^*G \text{ such that } |v| < \frac{2\pi}{\alpha_d}, \text{ and }  S u \neq 0 \right\rbrace \subset T_e^*G.
\end{equation}
Then $\exp_e : D \to \exp_e(D)$ is a smooth diffeomorphism and $\mathcal{C}(e):=G \setminus \exp_e(D) $ is a closed set with zero measure. Moreover, for all $\lambda \in D$, the curve $\gamma_\lambda(t)=\exp_e(t\lambda)$, with $t \in[0,1]$, is the unique minimizing geodesic joining its endpoints.
\end{lemma}
\begin{rmk}
By Lemma~\ref{l:geods}, geodesics $\gamma_\lambda$ with $\lambda =(u,v)$ satisfying $Su=0$ or $v=0$ are Euclidean straight lines $(tu,0)$, which are optimal for all times. This observation and Lemma~\ref{l:ingjd} imply that the cut time, that is the first time of loss of optimality, is
\begin{equation}
t_{\mathrm{cut}}(u,v) = \begin{cases} \frac{2\pi}{\alpha_d|v|} & \text{if } v \neq 0 \text{ and } Su\neq 0, \\
+\infty & \text{otherwise}. 
\end{cases}
\end{equation} 
For standard H-type groups $S=\mathbbold{1}_{2d}$, and Remark~\ref{eq:geodesicsclassical} yields $\mathrm{Cut}_e = \{\exp(t_{\mathrm{cut}}(\lambda)\lambda)\mid \lambda \in T_e^*G\} = \{(0,z)\mid z \in \R^{n-k}\}$ (the center of $G$). This fact was the main result proven in \cite{Mauricio}.
\end{rmk}
\begin{proof}
Since all geodesic are normal and $(G,d)$ is complete, each point of $G$ is reached by at least one minimizing normal geodesic $\gamma_\lambda :[0,1] \to G$, with $\lambda =(u,v) \in T_e^*G$. If $|v| > 2 \pi/\alpha_d$, and $Su \neq 0$, then $\gamma_\lambda$ is not abnormal with a conjugate time at $t_* = 2 \pi |v|/\alpha_d< 1$. Since these structures are real-analytic, $\gamma_\lambda$ does not contain any abnormal segment, hence cannot be minimizing after its first conjugate time (see \cite[Chapter 8]{nostrolibro} or \cite[Appendix A]{BR-grande1}). Thus $\gamma_\lambda$ is not minimizing on $[0,1]$. On the other hand, if $Su = 0$, for any value of $v$ we obtain the same abnormal geodesic (see Lemma~\ref{l:abnormals}). It follows that $\exp_e : \bar{D} \to G$ is onto (where $\bar D$ denotes the closure of $D$). Since also $\exp_e: D \to \exp_e(D)$ is trivially onto, $\mathcal{C}(e) = G \setminus \exp_e(D) \subseteq \exp_e(\partial D)$ has zero measure.

We now prove that, if $\lambda \in D$, then $\gamma_\lambda :[0,1] \to G$ is the unique minimizing geodesic joining its endpoints. Assume by contradiction that $\lambda=(u,v)$ and $\bar\lambda=(\bar u,\bar v) \in D$, with $\lambda \neq \bar{\lambda}$, are such that $\exp_e(\lambda) = \exp_e(\bar \lambda)$. Observe that the spectrum of $L_v$ is given by
\begin{equation}\label{eq:specLv}
\sigma(L_v) = \{0,\pm i |v| \alpha_1,\dots,\pm i |v|\alpha_d\}.
\end{equation}
Using the $z$ part of Lemma~\ref{l:geods}, and \eqref{eq:specLv}, we have $g(L_v)u \cdot u / |v|^2 \geq 0$, with equality if and only if $v=0$. This readily implies that $\bar{v} = v$.
From the $x$ equation of Lemma~\ref{l:geods}, we have that $\exp_e(u,v) = \exp_e(\bar{u},v)$ is equivalent to
\begin{equation}\label{eq:kernelf}
f(L_v) (u-\bar{u}) = 0.
\end{equation}
If $v=0$, and since $f(0) = 1$, \eqref{eq:kernelf} implies $u = \bar{u}$, which is a contradiction. Hence we proceed assuming $v \neq 0$. Since $|v| < 2\pi/|\alpha_d|$, we have that $\alpha_j |v| < 2\pi$  for all $j=1,\dots,d$. It is then easy to check that $f(\pm i |v| \alpha_j) \neq 0$. This observation, \eqref{eq:kernelf} and \eqref{eq:specLv} imply that $u - \bar{u} \in \ker L_v$. Using again that $f(0) = 1$, \eqref{eq:kernelf} implies $u = \bar{u}$, which is a contradiction. Thus $\exp_e :D \to \exp_e(D)$ is a bijection.

To conclude, no $\lambda \in D$ can be a critical point for $\exp_e$, since the Jacobian determinant of Lemma~\ref{l:jacobian} is strictly positive for $|v|< 2\pi/\alpha_d$ and $Su \neq 0$.
\end{proof}

The next corollary formalizes the fact that Definition~\ref{d:MCP} for the measure contraction property makes sense in the setting of generalized H-type groups.

\begin{cor}\label{c:good}
For any $x \in G$, consider the zero measure set $\mathcal{C}(x):= \tau_x (\mathcal{C}(e))$, where $\tau_x :G \to G$ is the left-translation. There exists a measurable map $\Phi^x : G \setminus \mathcal{C}(x) \times [0,1] \to G$,
\begin{equation}
\Phi^x(y,t) := \tau_x \exp_{e}(t \exp_e^{-1}(\tau_x^{-1} y)),
\end{equation}
such that $\Phi^x(y,t)$ is the unique minimizing geodesic joining $x$ with $y$.
\end{cor}

\section{Proof of the main result}

To prove Theorem~\ref{t:htype}, we begin with a simplified version of \cite[Lemma 2.6]{Juillet}. The argument will be useful to prove further inequalities appearing in generalized H-type groups.

\begin{lemma}\label{l:ineq}
Let $g(x):= \sin(x) - x \cos(x)$. Then, for all $x \in (0,\pi)$ and $t \in [0,1]$,
\begin{equation}
g(t x) \geq t^N g(x), \qquad \forall N \geq 3.
	\end{equation}
\end{lemma}
\begin{proof}
The condition $N \geq 3$ is necessary, as $g(x) = x^3/3 + O(x^4)$. It is sufficient to prove the statement for $N=3$. The cases $t=0$ and $t=1$ are trivial, hence we assume $t \in (0,1)$. By Gronwall's Lemma the above statement is implied by the differential inequality
\begin{equation}
g'(s) \leq  3 g(s)/s, \qquad s \in (0,\pi).
\end{equation}
In fact, it is sufficient to integrate the above inequality on $[tx,x] \subset (0,\pi)$ to prove our claim. The above inequality reads
\begin{equation}
f(s):=(3-s^2) \sin (s)-3 s \cos (s) \geq 0, \qquad s \in (0, \pi).
\end{equation}
To prove it, we observe that $f(0) = 0$ and $f'(s) =  s (\sin (s)-s \cos (s)) \geq 0$ on $[0,\pi]$.
\end{proof}
The next inequality is required to deal with the extra terms appearing in the case of generalized H-type groups, due to the fact that they have corank larger than one.
\begin{lemma}\label{l:ineq2}
Let $f(x):= x - \sin(x)$. Then, for all $x \in (0,2\pi)$ and $t \in [0,1]$,
\begin{equation}
f(t x) \geq t^N f(x), \qquad \forall N \geq 3.
\end{equation}
\end{lemma}
\begin{proof}
Similarly as in the proof of Lemma~\ref{l:ineq}, and by Gronwall's Lemma applied to the interval $[tx,x] \subset (0,2\pi)$, it is sufficient to prove the differential inequality
\begin{equation}
f'(s) \leq  3 f(s)/s, \qquad s \in (0,2\pi).
\end{equation}
The above inequality reads
\begin{equation}
h(s):=2s - 3 \sin(s) + s \cos(s) \geq 0, \qquad s \in (0,2\pi).
\end{equation}
It is sufficient to observe that $h'(0)=h''(0)=0$ and $h'''(s) = s \sin(s) \geq 0$ on $[0,2\pi]$.
\end{proof}
\begin{cor}\label{c:jacobineq}
For all $(u,v) \in D$, we have the following inequality
\begin{equation}
\frac{J(t u, t v)}{J(u,v)} \geq t^{2(n-k)}, \qquad \forall t \in [0,1].
\end{equation}
\end{cor}
\begin{proof}
Apply Lemmas \ref{l:ineq} and \ref{l:ineq2} to the explicit expression of $J$ from Lemma~\ref{l:jacobian}, and then use the standard inequality $\sin(t x) \geq t \sin(x)$, valid for all $x \in [0,\pi]$ and $t \in [0,1]$.
\end{proof}

\subsection{Proof of Theorem~\ref{t:htype}}
We are now ready to prove Theorem~\ref{t:htype}.
%\begin{theorem}
%Let $(G,d,\mu)$ be a generalized H-type Carnot group of rank $k$ and dimension $n$. Then it satisfies the $\mathrm{MCP}(K,N)$ if and only if $K \leq 0$ and $N \geq k+3(n-k)$, the latter integer being its geodesic dimension.
%\end{theorem}
Observe that the integer $k+3(n-k)$ coincides with the geodesic dimension of the Carnot group, by Lemma~\ref{l:fact}.

{\bf Step 1.} We first prove that the $\mathrm{MCP}(0,N)$ holds for $N\geq k+3(n-k)$. By left-translation, it is sufficient to prove the inequality \eqref{eq:MCP} for the homothety with center equal to the identity $e=(0,0)$. Let $\Omega$ be a measurable set with $0<\mu(\Omega)<+\infty$. 

By Lemma~\ref{l:ingjd}, up to removing a set of zero measure, $\Omega = \exp_e(A)$ for some $A \subset D \subset T_e^*G$. On the other hand, by Corollary~\ref{c:good}, we have
\begin{equation}
\Omega_t = \exp_e(t A), \qquad \forall t \in [0,1],
\end{equation}
where $t A $ denotes the set obtained by multiplying by $t$ any point of the set $A \subset T_e^*G$ (an Euclidean homothety). Thus, for all $t \in [0,1]$ we have
\begin{align}
\mu(\Omega_t) & = \int_{\Omega_t} d \mu = \int_{t A} J(u,v) du dv \\
& = t^{n}  \int_{A} J(t u, tv) du dv \geq  t^{k+3(n-k)}  \int_{A} J(u,v) d u d v = t^{k+3(n-k)} \mu(\Omega),
\end{align}
where we used Corollary~\ref{c:jacobineq}. In particular $\mu(\Omega_t) \geq t^N \mu(\Omega)$ for all $N \geq k+3(n-k)$.

{\bf Step 2.}
Fix $\varepsilon>0$. Let $N=k+3(n-k)$. We prove that the $\mathrm{MCP}(0,N-\varepsilon)$ does not hold. Let $\lambda = (u,0) \in D$. By Lemma~\ref{l:jacobian}, and recalling that $J > 0$ on $D$, we have 
\begin{equation}
J(t u,0)= t^{2(n-k)} J(u,0) < t^{2(n-k)-\varepsilon} J(u,0), \qquad \forall t \in[0,1].
\end{equation}
By continuity of $J$ and compactness of $[0,1]$, we find an open neighborhood $A \subset D$ such that $J(t\lambda) < t^{2(n-k)-\varepsilon} J(\lambda)$, for all $t \in [0,1]$ and $\lambda \in A$. Arguing as in the first step, for $\Omega= \exp_e(A)$, we obtain
\begin{equation}
\mu(\Omega_t) < t^{n+2(n-k)-\varepsilon} \mu(\Omega) = t^{N-\varepsilon}\mu(\Omega), \qquad t \in [0,1].
\end{equation}

{\bf Step 3.}
To  prove that $\mathrm{MCP}(K,N)$ does not hold for $K>0$ and any $N >1$, we observe that spaces verifying this condition are bounded (see \cite[Theorem 4.3]{Ohta-MCP}), while $G \simeq \R^n$ is not.
Finally, assume that $(G,d,\mu)$ satisfies $\mathrm{MCP}(K,N)$ for some $K < 0$ and $N < k+3(n-k)$. Then the scaled space $(G,\varepsilon^{-1} d, \varepsilon^{-Q} \mu)$ (where $\varepsilon>0$ and $Q = k+2(n-k)$ is the Hausdorff dimension of $(G,d)$) verifies $\mathrm{MCP}(\varepsilon^2 K,N)$ for \cite[Lemma 2.4]{Ohta-MCP}. But the two spaces $(G,d,\mu)$ and $(G,\varepsilon^{-1} d, \varepsilon^{-Q} \mu)$ are isometric through the dilation $\delta_\varepsilon(x,z):=(\varepsilon x, \varepsilon^2 z)$. In particular $(G,d,\mu)$ satisfies the $\mathrm{MCP}(\varepsilon K,N)$ for all $\varepsilon>0$, that is
\begin{equation}
\mu(\Omega_t) \geq \int_{\Omega} t \left[\frac{s_{\varepsilon K}(t d(x,z) / \sqrt{N-1})}{s_{\varepsilon K}(d(x,z) / \sqrt{N-1})}\right]^{N-1} d \mu(z), \qquad \forall t \in[0,1].
\end{equation}
Taking the limit for $\varepsilon \to 0^+$, we obtain that $(G,d,\mu)$ satisfies the $\mathrm{MCP}(0,N)$ with $N< k+3(n-k)$, but this is false as a result of the previous step.
\hfill $\qed$

%% file: MCP-Htype-v4.bbl
\begin{thebibliography}{10}

\bibitem{ABB-Hausdorff}
A.~Agrachev, D.~Barilari, and U.~Boscain.
\newblock On the {H}ausdorff volume in sub-{R}iemannian geometry.
\newblock {\em Calc. Var. Partial Differential Equations}, 43(3-4):355--388,
  2012.

\bibitem{nostrolibro}
A.~Agrachev, D.~Barilari, and U.~Boscain.
\newblock Introduction to {R}iemannian and sub-{R}iemannian geometry ({L}ecture
  {N}otes). http://webusers.imj-prg.fr/~davide.barilari/notes.php.
\newblock v20/11/16.

\bibitem{ABR-Curvature}
A.~{Agrachev}, D.~{Barilari}, and L.~{Rizzi}.
\newblock {Curvature: a variational approach}.
\newblock {\em Memoirs of the AMS (in press)}.

\bibitem{AL-transport}
A.~Agrachev and P.~Lee.
\newblock Optimal transportation under nonholonomic constraints.
\newblock {\em Trans. Amer. Math. Soc.}, 361(11):6019--6047, 2009.

\bibitem{AAAopenproblems}
A.~A. Agrachev.
\newblock Some open problems.
\newblock In {\em Geometric control theory and sub-{R}iemannian geometry},
  volume~5 of {\em Springer INdAM Ser.}, pages 1--13. Springer, Cham, 2014.

\bibitem{Rigot1}
L.~Ambrosio and S.~Rigot.
\newblock Optimal mass transportation in the {H}eisenberg group.
\newblock {\em J. Funct. Anal.}, 208(2):261--301, 2004.

\bibitem{Mauricio}
C.~Autenried and M.~G. Molina.
\newblock The sub-{R}iemannian cut locus of {$H$}-type groups.
\newblock {\em Math. Nachr.}, 289(1):4--12, 2016.

\bibitem{BKS16}
Z.~M. {Balogh}, A.~{Krist{\'a}ly}, and K.~{Sipos}.
\newblock {Geometric inequalities on Heisenberg groups}.
\newblock {\em ArXiv e-prints}, May 2016.

\bibitem{BKS17}
Z.~M. {Balogh}, A.~{Krist{\'a}ly}, and K.~{Sipos}.
\newblock {Jacobian determinant inequality on Corank 1 Carnot groups with
  applications}.
\newblock {\em ArXiv e-prints}, Jan. 2017.

\bibitem{BBN16}
D.~{Barilari}, U.~{Boscain}, and R.~W. {Neel}.
\newblock {Heat kernel asymptotics on sub-Riemannian manifolds with symmetries
  and applications to the bi-Heisenberg group}.
\newblock {\em Annales de la Faculté des Sciences de Toulouse (in press)}.

\bibitem{BR-popp}
D.~Barilari and L.~Rizzi.
\newblock A formula for {P}opp's volume in sub-{R}iemannian geometry.
\newblock {\em Anal. Geom. Metr. Spaces}, 1:42--57, 2013.

\bibitem{BR-grande1}
D.~{Barilari} and L.~{Rizzi}.
\newblock {Sub-Riemannian interpolation inequalities: ideal structures}.
\newblock {\em ArXiv e-prints}, May 2017.

\bibitem{BLU-Stratified}
A.~Bonfiglioli, E.~Lanconelli, and F.~Uguzzoni.
\newblock {\em Stratified {L}ie groups and potential theory for their
  sub-{L}aplacians}.
\newblock Springer Monographs in Mathematics. Springer, Berlin, 2007.

\bibitem{CM}
F.~Cavalletti and A.~Mondino.
\newblock Optimal maps in essentially non-branching spaces.
\newblock {\em Commun. Contemp. Math.}, 19(6):1750007, 27, 2017.

\bibitem{FJ}
A.~Figalli and N.~Juillet.
\newblock Absolute continuity of {W}asserstein geodesics in the {H}eisenberg
  group.
\newblock {\em J. Funct. Anal.}, 255(1):133--141, 2008.

\bibitem{FR}
A.~Figalli and L.~Rifford.
\newblock Mass transportation on sub-{R}iemannian manifolds.
\newblock {\em Geom. Funct. Anal.}, 20(1):124--159, 2010.

\bibitem{Juillet}
N.~Juillet.
\newblock Geometric inequalities and generalized {R}icci bounds in the
  {H}eisenberg group.
\newblock {\em Int. Math. Res. Not. IMRN}, (13):2347--2373, 2009.

\bibitem{Kaplan}
A.~Kaplan.
\newblock Fundamental solutions for a class of hypoelliptic {PDE} generated by
  composition of quadratic forms.
\newblock {\em Trans. Amer. Math. Soc.}, 258(1):147--153, 1980.

\bibitem{DMOPV-Sard}
E.~Le~Donne, R.~Montgomery, A.~Ottazzi, P.~Pansu, and D.~Vittone.
\newblock Sard property for the endpoint map on some {C}arnot groups.
\newblock {\em Ann. Inst. H. Poincar\'e Anal. Non Lin\'eaire},
  33(6):1639--1666, 2016.

\bibitem{LR17}
A.~Lerario and L.~Rizzi.
\newblock How many geodesics join two points on a contact sub-{R}iemannian
  manifold?
\newblock {\em J. Symplectic Geom.}, 15(1):247--305, 2017.

\bibitem{LV-Ricci}
J.~Lott and C.~Villani.
\newblock Ricci curvature for metric-measure spaces via optimal transport.
\newblock {\em Ann. of Math. (2)}, 169(3):903--991, 2009.

\bibitem{montgomerybook}
R.~Montgomery.
\newblock {\em A tour of subriemannian geometries, their geodesics and
  applications}, volume~91 of {\em Mathematical Surveys and Monographs}.
\newblock American Mathematical Society, Providence, RI, 2002.

\bibitem{Ohta-MCP}
S.-i. Ohta.
\newblock On the measure contraction property of metric measure spaces.
\newblock {\em Comment. Math. Helv.}, 82(4):805--828, 2007.

\bibitem{Rifford}
L.~Rifford.
\newblock Ricci curvatures in {C}arnot groups.
\newblock {\em Math. Control Relat. Fields}, 3(4):467--487, 2013.

\bibitem{rifford2014sub}
L.~Rifford.
\newblock {\em Sub-{R}iemannian geometry and optimal transport}.
\newblock SpringerBriefs in Mathematics. Springer, Cham, 2014.

\bibitem{RT-MorseSard}
L.~Rifford and E.~Tr{\'e}lat.
\newblock Morse-{S}ard type results in sub-{R}iemannian geometry.
\newblock {\em Math. Ann.}, 332(1):145--159, 2005.

\bibitem{Rigot2}
S.~Rigot.
\newblock Mass transportation in groups of type {$H$}.
\newblock {\em Commun. Contemp. Math.}, 7(4):509--537, 2005.

\bibitem{MCPc1}
L.~Rizzi.
\newblock Measure contraction properties of {C}arnot groups.
\newblock {\em Calc. Var. Partial Differential Equations}, 55(3):Art. 60, 20,
  2016.

\bibitem{S-I}
K.-T. Sturm.
\newblock On the geometry of metric measure spaces. {I}.
\newblock {\em Acta Math.}, 196(1):65--131, 2006.

\bibitem{S-II}
K.-T. Sturm.
\newblock On the geometry of metric measure spaces. {II}.
\newblock {\em Acta Math.}, 196(1):133--177, 2006.

\end{thebibliography}
